\documentclass[11pt, a4paper]{article}
\usepackage{amsfonts,amssymb,amsmath,amscd,latexsym,makeidx,theorem,mathtools}

\usepackage{hyperref}
\usepackage{color}

\newtheorem{thm}{Theorem}[section]
\newtheorem{thmx}{Theorem}

\newtheorem{lem}[thm]{Lemma}
\newtheorem{prop}[thm]{Proposition}
\newtheorem{cor}[thm]{Corollary}
\newtheorem{rem}[thm]{Remark}

\newtheorem{defn}{Definition}[section]

\newenvironment{proof}{\noindent\emph{Proof.}}{\hfill$\square$\medskip}

\newcommand{\D}{\Delta}
\newcommand{\vp}{\varphi}

\newcommand{\R}{\mathbb{R}}
\newcommand{\N}{\mathbb{N}}

\newcommand{\ve}{\varepsilon}

\newcommand{\pa}{\partial}

\newcommand{\Oo}{\mathcal{O}}
\newcommand{\HH}{\mathcal{H} }

\def\CC{\leavevmode\setbox0=\hbox{h}\dimen=\ht0 \advance \dimen by-1ex\rlap{\raise.9\dimen\hbox{\kern .15em \char'27}} C}

\newcommand{\mres}{\mathbin{\vrule height 1.6ex depth 0pt width
0.13ex\vrule height 0.13ex depth 0pt width 1.3ex}}

\title{Blow-up Analysis of Stationary Solutions to a  Liouville-Type Equation in  $3$-D} 
\author{Francesca Da Lio\thanks{Department of Mathematics, ETH Z\"urich, R\"amistrasse 101, 8092 Z\"urich, Switzerland.} and Ali Hyder\thanks{TIFR Centre for Applicable Mathematics, Sharadanagar, Yelahanka New Town, Bangalore 560065, India.}}

\begin{document}

\maketitle
\begin{abstract}
In this paper we study the asymptotic behavior of sequences of stationary  weak solutions  to  the following Liouville-type  equation
\begin{align}\label{eq-abs} -\D u=e^u\quad\text{in }\Omega,\end{align}
where $\Omega\subseteq\R^3$ is an open set.  By improving the partial regularity estimates obtained by the first author in \cite{DaLio} for the equation \eqref{eq-abs} we succeed in   performing a blow-up analysis without  Morrey-type assumptions on the solutions $u$ and on the nonlinearity $e^u.$  \par

  \end{abstract}
 
 \medskip
{\noindent{\small {\bf  MSC 2020.} 35J61, 35B44, 35B65, 35D30. }

 \tableofcontents 

\section{Introduction and main results}  
 This paper is concerned with compactness results for sequences of weak solutions to the following Liouville-type  equation
\begin{align}\label{eq-1} -\D u=e^u\quad\text{in }\Omega,\end{align}
where $\Omega\subseteq\R^3$ is an open set.\par

  The equation  \eqref{eq-1} is the Euler-Lagrange of the following energy functional
\begin{equation}\label{energy}
E(u)=\frac{1}{2}\int_{\Omega}|\nabla u|^2 dx-\int_{\Omega} e^u dx.\end{equation}
 A weak solution of
 \eqref{eq-1} is a function $u$  which is a solution of  \eqref{eq-1}  in the sense of distributions, belongs to $H^1(\Omega)$ and such that $e^{u}\in L^1(\Omega)$
 
 We recall that in dimension $m=2$ the equation \eqref{eq-1} has a geometrical meaning. More precisely  if $(\Sigma, g_0)$  is a smooth, closed Riemann surface with  Gaussian curvature $K_{g_0}$, an easy computation shows that a function $K(x)$  is the Gaussian curvature for some {metric $g=e^{2u}g_0$ conformally equivalent} to the  metric $g_0$  with $u \colon\Sigma \to\R, $ if and only if there exists a solution $u = u(x)$ of
\begin{equation}\label{ld2}-\Delta_{g_0} u=K e^{2u}-K_{{g_0}}~~\text{on } \Sigma,\end{equation}
where $\Delta_{g_0}$ is the Laplace Beltrami operator on $(\Sigma, g_0)$\,, (see e.g. \cite{ch} for more details).
In particular, when $\Sigma =\R^2$ or $\Sigma=S^2$ equation \eqref{ld2} reads respectively
\begin{equation}\label{ld2r} -\Delta u=K e^{2u} ~~\text{on }\R^2,\end{equation} and
\begin{equation*}-\Delta_{S^2} u=K e^{2u}-1~~\text{on } S^2.\end{equation*}\par
 In $2$ dimension the equation \eqref{eq-1} is critical   and it was proved by Br\'ezis and Merle in \cite{BM} that any weak solution
 with $e^{u}\in L^1(\Omega)$ is smooth. In  \cite{BM} the authors also showed the following remarkable 
   blow-up result:

\begin{thm}[Theorem 3, \cite{BM}]\label{thbm}
Let $ \Omega$ be an open subset of $\R^2$ and $(u_k)\subset L^1(\Omega)$ be a sequence of solutions to \eqref{ld2r} satisfying for some $1<p\le  \infty$, $K_k\ge 0$,
$\|K_k\|_{L^p}\le C_1\,,$ and $\|e^{u_k}\|_{L^{p'}}\le C_2\,.$ Then up to subsequences the following alternatives hold:
either $(u_k)$ is bounded in $L^{\infty}_{loc}(\Omega)$, or $u_k(x)\to -\infty$  uniformly on compact subsets of $\Omega$, or
there is a finite nonempty set $B=\{a_1,\ldots,a_N\}\subset\Omega$  (concentration set) such that
$u_k(x)\to -\infty$ on compact subsets of $\Omega\setminus B$. In addition in this last case
$K_k e^{2u_k}$ converges in the sense of measure on $\Omega$ to $\sum_{i=1}^N\alpha_i\delta_{a_i},$ with 
$\alpha_i\ge{\frac{2\pi}{{p^{\prime}}} }\,.$~~\hfill$\Box$
\end{thm}
In dimensions $m\ge 3$  the equation \eqref{eq-1}  becomes supercritical and it has singular weak solutions. 
In the case of supercritical elliptic equations
it turns out in general that the class of weak solutions is too large to develop a partial regularity theory and for this reason  one  restricts the attention to  the subclass of {\em stationary weak solutions}. A weak solution is said to be {\em stationary} if it is a
critical point of the associated energy under perturbations of the domain, namely it satisfies
\begin{equation}\label{deren} \frac{d}{dt} E(u(x+t X))_{|_{t=0}}=0\,,\end{equation}
for all smooth vector fields $X$ with compact support in $\Omega$.

This notion of stationary solutions  has already been successfully used in partial regularity results
for harmonic maps  \cite{Hel,E,RS}, minimal surfaces, semi-linear equations with power type non-linearity \cite{Pacard}.  
\par
  
A consequence of the stationary assumption \eqref{deren} is that for any smooth vector field $X\colon\Omega\to\R^3$  the following identity holds
 
\begin{equation}\label{stationary}
\int_{\Omega} \left[ \frac{\partial u}{\partial x_i}\frac{\partial u}{\partial x_k}\frac{\partial X^k}{\partial x_i}-\frac{1}{2}|\nabla u|^2\frac{\partial
X^i}{\partial x_i}+e^u \frac{\partial X^i}{\partial x_i} \right] dx=0\,.
\end{equation}
The identity \eqref{stationary}  can be also understood as a conservation law (see e.g \cite{PR}, \cite{DaLio2}).\par

In some situations the  identity   \eqref{stationary} can be converted into a monotonicity formula. 
This monotonicity formula may imply that the solution $u$ belongs to some Morrey type space ${\cal M}$, which is much smaller than the original space to which $u$ belongs. In the good cases, replacing the original space   by ${\cal M}$ makes the problem critical and this allows one to obtain a partial regularity result for the stationary solutions (see for instance \cite{Hel,E} for harmonic maps, and \cite{Pacard} when the nonlinearity is $u^\alpha$ with $\alpha$ greater than the critical exponent )\,.\par

The monotonicity  formula also  permits to get information on the structure of the singular set of a weak solution and of the concentration set  corresponding  to a sequence of weak solutions (see e.g. \cite{Lin}).
In the case of  \eqref{eq-1} in $3$-dimension the stationary assumption was converted into the following identity (see  \cite{DaLio}):
 \begin{equation}\label{rela-stationary}
   \frac{d}{dr}\left[\frac{1}{r}\int_{B_r(x_0)}(|\nabla u|^2-6 e^u)dx \right]=\frac{2}{r}\int_{\partial B_r(x_0)}
  |\frac{\partial u}{\partial r}|^2 dx-\frac{4}{r}\int_{\partial B_r(x_0)}  e^udx\,,
  \end{equation}
(for $x_0\in\Omega$ and $r>0$ small enough), which does not seem to provide neither any monotonicity information nor any uniform bound  for the  terms 
 $\frac{1}{r}\int_{ B_r(x_0)}|\nabla u|^2 dx$ and   $\frac{1}{r}\int_{ B_r(x_0)}e^{u} dx\,.$ 
 We mention that  in \cite{DaLio} the first author   presented an alternative approach to the partial regularity theory  to \eqref{eq-1} in absence of monotonicity and Morrey type estimates. 
Such an  approach was inspired by the technique introduced by Fang-Hua Lin and Tristan Rivi\`ere in \cite{LT-1} in the context of Ginzburg-Landau equations.  
We recall here the main result obtained  in \cite{DaLio}:

  \begin{thmx}[Theorem 1.1, \cite{DaLio}]\label{regularityDaLio}
Let  $u\in H^1(\Omega)$ be a stationary solution of \eqref{eq-1}, such that $e^u\in L^1(\Omega)$. Then there exists an open set $\Oo\subset\Omega$ such that
  $$u\in C^{\infty}(\Oo)~~ \mbox{and} ~~{\cal{H}}_{dim}(\Omega\setminus\Oo)\le 1\,$$
  where ${\cal{H}}_{dim}$ denotes the dimensional Hausdorff measure\,.
  \end{thmx} 

Given $x_0\in\Omega$ and $0<r< d(x_0,\partial\Omega)$    we introduce the following energy 
  \begin{equation}\label{energy2}
  {{E}}_{x_0,r}(u)=\frac{1}{r}\int_{B_r(x_0)}|\nabla u|^2 dx+\frac{1}{r}\int_{B_r(x_0)} e^u dx\,
  \end{equation}
  and   set
  $$
 \bar u_{x_0}(r):=\frac{1}{|\partial B_{r}(x_0)|}\int_{\partial B_{r}(x_0)} u(y) dy\,.
  $$\par
 The key result to prove Theorem \ref{regularityDaLio}  was the following assertion about the energy (\ref{energy2}). %\par

\begin{thmx}[Theorem 2.1, \cite{DaLio}] \label{thm-DaLio} There exists $\eta>0$ small such that if 
 \begin{align}\label{small-energy-1} E_{x_0,2r}(u)\leq\eta\quad\text{and } \bar u_{x_0}(r)E_{x_0,2r}(u) \leq\eta,  
 \end{align}   for some $x_0\in\Omega$ and $0<r<dist(x_0, \pa\Omega)/2$, then $u$ is regular in a small neighborhood of $x_0$.  Consequently, there exists an open set $\Oo\subset\Omega$  such that  the Hausdorff dimension of $\Omega\setminus\Oo$  is at most $1$, and  $u\in C^\infty(\Oo)$.  \end{thmx}

 Actually a simple scaling argument  shows that the second condition of \eqref{small-energy-1} is not necessary in the above theorem. More precisely, if we set $\tilde u(x)=u(x_0+rx)+2\log r$, then $\tilde u$ satisfies \eqref{eq-1} in $B_2$ with the energy bound $E_{0, 2}(\tilde u)=E_{x_0,r}(u)\leq\eta$. As $\eta>0$ is small, one easily gets that $\bar {\tilde u}_0(1)<0$, and hence $\tilde u$ also satisfies the second condition of \eqref{small-energy-1} (see Lemma \ref{negativeaverage} for the details).\par
 As far as the  asymptotic behavior  of stationary solutions in dimensions  $m\ge 3$ is concerned, we mention the paper by  Bartolucci,   Leoni, and Orsina \cite{BLO}. In \cite{BLO}  the standing assumption is the following uniform bound
 \begin{equation}\label{unifbound}
\sup_{\substack{x\in \Omega\\ r>0}}\frac{1}{r^{m-2}}\int_{B_r(x)\cap\Omega}e^{u_n} dx\le C,
\end{equation}
namely $e^{u_n}$ belongs to the Morrey space of index $m/2$. Furthermore, 
under the additional ad-hoc assumption that
\begin{equation}\label{montass}
r\mapsto \Phi_{n,x}(r):= \frac{1}{r^{m-2}}\int_{B_r(x)\cap\Omega}e^{u_n} dx,\end{equation} is monotone nondecreasing, they show that 
if  $e^{u_n}\stackrel{\ast}{\rightharpoonup}\mu$ as $n\to +\infty$ weakly as Radon measures and the $m-2$ Hausdorff measure of the singular set $\Sigma$ of the sequence $(u_n)$ is strictly positive (see definition in \eqref{Singset}), then $\Sigma$ is a $m-2$ rectifiable set with ${\mathcal{H}}^{m-2}(\Sigma)<+\infty$   and    
$\mu=\alpha(x){\mathcal{H}}^{m-2}\mres \Sigma,$ for some density function $\alpha(x)\ge 4\pi$ for ${\mathcal{H}}^{m-2}$-a.e. $x\in\Sigma.$ All the results in \cite{BLO} are proved under the assumption that $\Omega$ is a bounded subset of $\R^n$.\par

The main aim of this 
 paper is to study the asymptotic behaviour of  weak  solutions to \eqref{eq-1} in dimension $m=3$ without  assuming the  Morrey type bound \eqref{unifbound} and the monotonicity condition \eqref{montass}.\par 
 Our approach and the   results we obtain in this paper should also hold for the more general class of equations of the form $-\Delta u = V(x) e^u$ where $V(x)$ is some
smooth given potential. For the sake of simplicity, we have chosen to focus our attention on the case where $V\equiv 1$ in order to  keep the technicalities as low as possible and make the paper more ``readable".\par
 We first introduce some preliminary definition and notations.
We recall the definition of regular points of a weak  solution to \eqref{eq-1}.
 
 \begin{defn} A point $x_0\in\Omega$ is said to be a regular point of $u$ if $u^+\in L^\infty(B_\ve(x_0))$ for some  $\ve>0$. Similarly, $x_0$ is a regular point for a sequence of solutions $(u_k)$ to \eqref{eq-1} if there exists $\ve>0$ (independent of $k$) such that  $$\limsup_{k\to\infty}\|u_k^+\|_{L^\infty(B_\ve(x_0))}<\infty.$$
\end{defn}

Notice that by definition, the set of regular points is an open set, and by the elliptic regularity theory, solutions are smooth around regular points. Moreover by the above remarks if $x_0$ is a regular point of $u$ then $0$ is regular point of $\tilde u(x)=u(x_0+x)$. The {\em singular set} of a sequence $(u_k)$ is defined  by (we do not require that $u_k^+\in L^\infty(B_\ve(x_0))$)
\begin{equation}\label{Singset}
S=\{x\in \Omega:~~\forall \varepsilon>0,~\limsup_{k\to\infty}\|u_k^+\|_{L^\infty(B_\ve(x_0))}=+\infty\}.
\end{equation}

We will show the following  refined version of Theorem \ref{thm-DaLio}, in the sense  that we give a more refined estimated of the Hausdorff measure of the singular set.

 \begin{thm}\label{thm-1}  
 There exists $\eta>0$ (independent of $u, x_0,r_0$) such that if  $E_{x_0,2r_0}(u)\leq\eta$ for some $x_0\in\Omega$ and $0<r_0<\frac12 dist(x_0,\pa\Omega)$ then there exists $C\geq 1$ ($C$ is independent of $u,x_0,r_0$) such that $$u(x)\leq CE_{x_0,2r_0}(u)+2\log\frac{1}{r_0}\quad\text{for every }x\in B_{\frac{ r_0}{4}}(x_0).$$  As a consequence,  $u$ is smooth in $\Oo$  for some open set $\Oo\subset\Omega$ with $\mathcal H^1(\Omega\setminus\Oo)=0$,  where $\mathcal H^s(A)$ denotes the $s$-dimensional Hausdorff measure  of $A$.\end{thm}

  Now we consider  a sequence of stationary solutions   $(u_k)\subset H^1(\Omega)$   to \eqref{eq-1} satisfying 
    \begin{align}\label{bound-int} \int_{\Omega} e^{u_k}dx\leq C.   \end{align}    
  
  \begin{thm}\label{thm-2} Let $(u_k)$ be a sequence of stationary solutions to \eqref{eq-1} satisfying \eqref{bound-int}. 
   Assume also  that         \begin{align}\label{bound-grad}  \int_{\Omega} |\nabla u_k|^2dx\leq C. \end{align}    Then   
one of the following holds: \begin{itemize} \item[(i)] (vanishing) $u_k\to-\infty$ locally uniformly in $\Omega$.    \item[(ii)] (compactness) There exists an open set $\mathcal O\subset\Omega$ with $\mathcal H^1(\Omega\setminus\mathcal O)=0$ such that, up to a subsequence,  for some   $u\in H^1(\Omega)$  we have that $u_k\to u$ in $ C^2_{loc}(\mathcal O)$, and $$e^{u_k}\to e^u \quad\text{in }L^1_{loc}(\Omega).$$  
\end{itemize} 
 \end{thm} 
We will call the set $\Sigma=\Omega\setminus\mathcal O$    {\em concentration set} and it  is the union of the following two sets

$$\Sigma_1=\left\{x\in\Omega:\limsup_{r\to0}\frac1r\mu_1(B_r(x))>0\right\},~~\Sigma_2:=\left\{ x\in \Omega: \liminf_{r\to0}\frac{\mu_2(B_r(x))}{r}=\infty\right\},$$  
where $\mu_1,\mu_2$ are Radon measures such that  $e^{u_k}\rightharpoonup \mu_1,\quad |\nabla u_k|^2\rightharpoonup \mu_2.$
\medskip

The above theorem is sharp in the sense that if \eqref{bound-grad} is not satisfied, then the theorem is not true. For instance  (see  \cite[example 5.3]{BLO}) the sequence of regular solutions  $$u_k(x)=\log\left(\frac{8k^2}{(1+k^2|\bar x|^2)^2}\right) ,\quad x=(\bar x, x_3)\in\R^2\times\R,$$   does not satisfy \eqref{bound-grad}, $\mathcal H^1(\Sigma)=1$ and   $u_k\to-\infty$ outside $\Sigma$ (here $\Sigma $ is the $x_3$ axis).
Let us also emphasize that if \eqref{bound-grad}  is violated, then not necessarily that the   singular set is non-empty.   For instance,  it is possible to have a sequence of regular  solutions $(u_k)$ in $B_1$  of the form $$u_k(x)=kx_1-k^2+o_k(1)\quad\text{in }B_1,$$  for which  $\sup_{B_1}u_k\to-\infty$, but $|\nabla u_k|=k+o(1)\to\infty$, see  end of Section \ref{section-2}.

In our next theorem we give  an improvement of  Theorem \ref{thm-1}  under a Morrey-type assumption on the gradient. 
 \begin{thm} \label{thm-3}  Let $u$ be a stationary solution to \eqref{eq-1}.   Assume that       \begin{align}\label{Morrey-bound}  \frac1r\int_{B_r(x_0)} |\nabla u|^2dx\leq C\quad\text{for every } x_0\in\Omega, 0<r<dist(x_0,\pa\Omega). \end{align}    Then $u\in C^\infty(\Omega\setminus\Sigma)$ for some  discrete  set $\Sigma\subset\Omega$  (possibly empty). Moreover, if $\Sigma\neq\emptyset$ then    for every $x_0\in\Sigma$ we have \begin{align}\label{beha-singu}u(x)=-2\log|x-x_0|+O(1)\quad\text{around }x_0,\end{align}   and  \begin{align}\label{int}\lim_{r\to0} \frac{1}{r}\int_{B_r(x_0)}e^{u}dx=8\pi.\end{align}  Furthermore,   for  each $x_0\in \Sigma$ there exists $\vp\in C^\infty(S^2)$  satisfying  \begin{equation}\label{phiS}-\D_{S^2} \vp+1=e^{2\vp}\quad\text{in }S^2,\end{equation} and    \begin{align}\label{homog}  \lim_{r\to0} (u(x_0+r\cdot )+2\log r-\log2)= 2\vp\quad\text{in } C^\ell(S^2),~~\forall \ell\in\N.   \end{align}   \end{thm}

Combining Theorem \ref{thm-2} and Theorem \ref{thm-3} we get  the following corollary. 

\begin{cor} \label{cor-1}Assume that we are in case $(ii)$ of Theorem \ref{thm-2}, and that the sequence $(u_k)$ satisfies the following stronger condition $$ \frac1r\int_{B_r(x_0)}|\nabla u_k|^2dx\leq M\quad\text{for every }B_r(x_0)\subset\Omega.$$   Then $u_k\to u$ in $H^1_{loc}(\Omega)$, and the concentration set $\Omega\setminus \mathcal O$ is discrete.  

\end{cor}

Concerning global solutions we have the following theorem.

\begin{thm} \label{thm-4}  Let $u$ be a stationary solution to  $$-\D u=e^{u}\quad\text{in }\R^3.$$   Assume that       \begin{align}\label{Morrey-bound-2}  \frac1r\int_{B_r(x_0)} |\nabla u|^2dx\leq C_1\quad\text{for every } x_0\in \R^3, \, r\in(0,\infty). \end{align}   If $u$ is regular in $\R^3$ then    \begin{align*} |u(x)+2\log(1+|x|)|\leq C\quad\text{in }\R^3,  \end{align*}  and if $u$ is singular  then, up to a translation  \begin{equation*}
u(x)=-2\log|x|+\log2+2\vp(\frac{x}{|x|})\quad\text{in }\R^3,\end{equation*} for some $\vp$ satisfying \eqref{phiS}.   \end{thm}

Let us mention here that the conclusions of Theorems \ref{thm-3} and \ref{thm-4} were obtained by Veron-Veron \cite{Veron} under a  different hypothesis on $u$, namely (simply $\Omega=B_1$ or $\R^3$) \begin{align}\label{hypV} u(x)\leq -2\log|x|+C\quad\text{in }\Omega\setminus\{0\}. \end{align} Notice that the  assumption \eqref{hypV} on $ u$ would imply \eqref{Morrey-bound}, see e.g. Lemma \ref{lem-useful-1}.  Regarding  the existence of singular solutions, R\'eba\"i \cite{Rebai}  proved existence of solutions to \eqref{eq-1} with finitely many singularities, and  by constructions,  they satisfy \ref{beha-singu}. 

Concerning the regularity and partial regularity of stable solutions to semilinear equations with exponential nonlinearity we refer the reader for instance to the paper \cite{W} and the references therein.   %In a forthcoming article we are going to study partial regularity results for stable solutions to semilinear equations . % { \color{red} what do you mean?}

\par
We conclude the introduction by mentioning that the main questions that still remain open are the extension of the partial regularity result obtained in \cite{DaLio} and the blow-up analysis of the current paper for the equation \eqref{eq-1} in dimension $n>3$ and the rectifiability of the concentration set for a sequence of stationary  solutions to \eqref{eq-1} without the ad-hoc monotonicity assumption \eqref{montass}.

 \section{Small energy regularity and proof of Theorem \ref{thm-1}  }   \label{section-2}
 In this section we prove Theorem \ref{thm-1} which is a refined version of Theorem \ref{thm-DaLio}.\par
 
  Though the proof of Theorem \ref{thm-1} is very similar to that of Theorem \ref{thm-DaLio},   for convenience we will give a sketch of it. The main difference in our arguments is that we avoid the use of the average quantity $\bar u _{x_0}(r)$.   Some of the   estimates contained in this section will be used in later sections as well.

We start with  the following  energy decay estimate   which is the main ingredient for the proof of our main results.  Before stating the theorem let us set   
\begin{equation}\label{eta1}
\eta_1:=\frac{1}{2}\int_{B_{2}(x_0)}e^u dx,\quad \eta_2:=\frac{1}{2}\int_{B_{2}(x_0)}|\nabla u|^2 dx.
\end{equation}
We may assume without restriction that $x_0=0$.
 \begin{thm} \label{thm-energy-decay}For a given $\gamma\in(0,1)$ there exists $\bar\eta_1\in(0,1)$  and $0<\underline \rho<\bar \rho\leq\frac12$ such that if  $$\eta_1:=\frac{1}{2}\int_{B_{2}(0)}e^u dx\leq\bar\eta_1,$$ then there exists $\rho\in[\underline\rho,\bar\rho]$   for which we have   $$E_{0, \rho}(u)\leq\gamma\, E_{0,2}(u). $$  \end{thm}
 
\par
Before proving Theorem \ref{thm-energy-decay} 
we show some preliminary results.
We split $u$ in $B_1$ as $u=v+w$ where \begin{align}\label{decompose} \left\{\begin{array}{ll} -\D v=e^u&\quad\text{in }B_1\\ v=0&\quad\text{on }\pa B_1, \end{array}\right. \quad  \left\{\begin{array}{ll} \D w=0 &\quad\text{in }B_1\\ w=u &\quad\text{on }\pa B_1 .\end{array}\right.  \end{align} 
As $w$ is harmonic, we have that the functions $e^w$ and $|\nabla w|^2$ are subharmonic. 

Hence, we have the following lemma (see also the proof of Lemma 2.1 in \cite{DaLio}):
 \begin{lem} \label{lem-w}The  functions $$\rho\mapsto \frac{1}{|B_\rho|}\int_{B_\rho}e^w dx, \quad \rho\mapsto \frac{1}{|B_\rho|}\int_{B_\rho}|\nabla w|^2 dx\quad \text{and } \rho\mapsto \frac{1}{|\pa B_\rho|}\int_{\pa B_\rho}|\nabla w|^2d\sigma ,$$  are monotone increasing with respect to $\rho\in (0,1)$.  \end{lem}

  From the maximum principle it follows that $v>0$ in $B_1$, and hence $w<u$ in $B_1$. Moreover,  \begin{align} \label{uvw} \int_{B_1}|\nabla u|^2dx=\int_{B_1}|\nabla v|^2dx+\int_{B_1}|\nabla w|^2dx  .\end{align}  \begin{lem} We have    \begin{align} \label{est-w-1}  e^{w(x)}\leq C \int_{B_1}e^{u(y) }dy\quad \text{for every }x\in B_\frac12,% \frac{1}{\rho }\int_{B_\rho}e^wdx\leq \rho^2\int_{B_1}e^wdx \leq  \rho^2 \int_{B_1}e^udx, \quad \rho\in (0,1) ,
 \end{align}  \begin{align} \label{est-w-2} \frac{1}{\rho }\int_{B_\rho}|\nabla w|^2dx\leq \rho^2 \int_{B_1}|\nabla w|^2dx \leq  \rho^2 \int_{B_1}|\nabla u|^2dx, \quad \rho\in (0,1) , \end{align} \begin{align}  \int_{\pa B_\rho}|\nabla w|^2d\sigma\leq 4\rho^2\int_{\pa B_\frac12}|\nabla w|^2d\sigma\leq C\rho^2\int_{B_1}|\nabla u|^2dx,\quad \rho\in (0,\frac12). \label{est-w-3}\end{align} \end{lem}

 \begin{proof} {\bf 1.} To prove \eqref{est-w-1} we  use the monotonicity of $$\rho\mapsto\frac{1}{|B_\rho|}\int_{B_\rho(x)} e^{w(y)}dy,$$ thanks to Lemma \ref{lem-w}, and  the fact that $w<u$ in $B_1$. Indeed, taking $\rho\downarrow 0$ we get that $$e^{w(x)}\leq\frac{1}{B_\frac14(x)}\int_{B_\frac14(x)}e^{w(y)}dy\leq\frac{1}{B_\frac14(x)}\int_{B_\frac14(x)}e^{u(y)}dy\leq C\int_{B_1} e^{u(y)}dy,$$ for every $x\in B_\frac 14$.

{\bf 2.} To prove \eqref{est-w-3} we  use Lemma \ref{lem-w}, \eqref{uvw} and the monotonicity of 
$\rho\mapsto \frac{1}{|\pa B_\rho|}\int_{\pa B_\rho}|\nabla w|^2d\sigma$.

 Indeed, for $\rho\in(0,\frac12)$ we have \begin{align*} \int_{\pa B_\rho}|\nabla w|^2d\sigma &\leq |\pa B_\rho| \frac{1}{|\pa B_\frac12|}\int_{\pa B_\frac12}|\nabla w|^2d\sigma \\ & \leq|\pa B_\rho| 2\int_\frac12^1\frac{1}{|\pa B_t|}\int_{\pa B_t}|\nabla w|^2d\sigma dt \\&\leq C\rho^2\int_{B_1}|\nabla w|^2dx\\ &\leq C\rho^2\int_{B_1}|\nabla u|^2dx. \end{align*}

The proof of \eqref{est-w-2} follows in a similar way. 
\end{proof}

   From the Green's representation for $v$, we see that ($C$ is a dimensional constant) 
   
   \begin{equation*}%\label{Greenv}
   v(x)\leq C\int_{B_1}\frac{1}{|x-y|} e^{u(y)}dy.
   \end{equation*}
    Therefore, using the  Minikowski's integral inequality $$\left(\int_{B_1} \left|\int_{ B_1} f(x,y) dy\right|^pdx\right)^\frac1p\leq \int_{B_1} \left(\int_{B_1} |f(x,y)|^pdx\right)^\frac1pdy,\quad 1\leq p<\infty,$$ one can get \begin{align}  \label{est-v-L2} \left(\int_{B_1} v^2dx\right)^\frac12\leq C \int_{B_1}e^u dx. \end{align} Consequently, by Chebyshev's inequality and by \eqref{eta1} we get 
    
    \begin{equation}\label{est-v-measure}  |\{x\in B_1: v(x)\geq\sqrt{\eta_1}\}|\leq C\eta_1. \end{equation}
    
    Combining the above estimates  we prove:  
    
  \begin{prop}\label{propo-e^u}
   For every $\rho\in (0,\frac12]$, $0<\eta_1<1$ and  $\lambda\geq 1$, we have  \begin{align*}\frac{1}{\rho }\int_{B_\rho}e^udx\leq \left( \frac{C e^\lambda \eta_1}{\rho} + C\rho^2     \right) \int_{B_1}e^udx +\frac{1}{\rho\lambda}  \int_{B_1}|\nabla u|^2dx.\end{align*} 
   \end{prop}
   
   \begin{proof}  Let $\rho\in (0,\frac12]$. For  $\lambda\geq 1$ we split $B_\rho$ into 
   \begin{equation*}A_1:=\{x\in B_\rho: v(x)\leq \sqrt{\eta_1}\}, \quad A_2:=\{x\in B_\rho: \sqrt{\eta_1}<v\leq\lambda\},\quad A_3:=\{x\in B_\rho: v\geq\lambda\}. 
    \end{equation*}
    By \eqref{est-w-1} we get that 
   \begin{equation*}\frac{1}{\rho}\int_{A_1} e^u dx\leq e^{\sqrt{\eta_1}} \frac1\rho \int_{B_\rho}e^wdx\leq e \rho^2 \int_{B_1}e^u dx. \end{equation*}
   Again by \eqref{est-w-1}, and together with \eqref{est-v-measure}
    \begin{align*} \frac{1}{\rho}\int_{A_2}e^udx\leq \frac1\rho \left( \sup_{B_\rho}e^w  \right) e^\lambda |A_2|\leq \frac{C e^{\lambda} \eta_1}{\rho} \int_{B_1}e^{u}dx.\end{align*} 
    Finally, 
    \begin{align*}   \frac1\rho \int_{A_3} e^udx\leq \frac{1}{\rho\lambda}\int_{B_\rho} v e^udx\leq \frac{1}{\rho\lambda}\int_{B_1} v (-\D v)dx= \frac{1}{\rho\lambda}\int_{B_1} |\nabla v|^2dx\leq   \frac{1}{\rho\lambda}\int_{B_1} |\nabla u|^2dx, \end{align*} thanks to \eqref{uvw}. 
   
   We can conclude the proof of the proposition \ref{propo-e^u}. 
   \end{proof}

% \subsection{Gradient estimate} 
 
 \begin{lem}\label{lem-2.5} Setting $$u^+:=\max\{u,0\},\quad u^-:=\min\{u,0\},$$ 
 we have
\begin{eqnarray}
\int_{B_1}|\nabla u^+|^2dx &\leq& 2 \int_{B_{2}}u^+e^u +8\int_{B_2}e^udx \label{uplus}\\
 \int_{B_1}|\nabla u^-|^2dx&\leq & 4\int_{B_{2}}(u^-)^2dx\label{uminus}\\
 \int_{B_1}u^+e^u dx&\leq & 2 \int_{B_{2}} |\nabla u^+|^2dx +2\int_{B_2}e^udx.\label{uplusexp}
 \end{eqnarray}
  \end{lem} 
 
 \begin{proof} Set \begin{align*}\vp(x)=\left\{\begin{array}{ll}  1&\quad\text{for }|x|\leq 1\\ 2-|x| &\quad\text{for }1\leq|x|\leq 2.\end{array}\right.  \end{align*} Then $u\vp\in H^1_0(B_2) $. Taking $u^+\vp^2$ as a test function in  the weak formulation of \eqref{eq-1}, we obtain  $$\int_{B_2}\nabla u\cdot\nabla (u^+\vp^2)dx=\int_{B_2}u^+\vp^2e^udx.$$ The left hand side can be estimated  from below by   \begin{align*}&\int_{B_2} |\nabla u^+|^2\vp^2dx+2\int_{B_2}(\frac{1}{\sqrt 2}\vp\nabla u^+)\cdot(\sqrt 2u^+\nabla\vp)dx \\ &\geq \frac12\int_{B_2} |\nabla u^+|^2\vp^2dx-2\int_{B_2}(u^+)^2|\nabla \vp|^2dx.\end{align*}  Since $|\nabla \vp|\leq1$, we obtain $$\int_{B_1}|\nabla u^+|^2dx\leq 2\int_{B_2}u^+e^udx+4\int_{B_2}(u^+)^2dx\leq 2\int_{B_2}u^+e^udx+8\int_{B_2}e^udx,$$ where the last inequality follows from $t^2\leq 2e^t$ for $t>0$. This proves \eqref{uplus}. 
 
 In a similar way one gets \eqref{uminus} and \eqref{uplusexp}. 
 \end{proof}
 
 As a consequence of \eqref{uplusexp} of the above lemma we obtain     \begin{lem}\label{entropy}We have  \begin{align}\label{12} \int_{B_1}|\nabla^2 v|dx &\leq C\int_{B_1}|\D v|\log(2+|\D v|)dx\notag\\ &\leq C\left(\int_{B_2}e^u dx+\int_{B_2}|\nabla u|^2dx\right).\end{align} 
   \end{lem}
   We refer to Section 5.2 in \cite{Ste} for the properties of  Orlicz spaces $L^1 Log L^1(\R^n)$ and Theorem 4.21 in the the Lecture Note \cite{RivN} for the proof   of the   Lemma \ref{entropy}.
   \par
   \medskip

   Now we recall  the definition of the weak $L^2$ space (or Marcinkievicz space $L^{2,\infty})\,,$ (see e.g. \cite{Ste})\,. 
  The space $L^{2,\infty}(\Omega)$ is   defined as the space of functions $f\colon \Omega\to\R$ such that
   $$\sup_{\lambda\in\R}\lambda|\{x:|f|(x)\ge\lambda\}|^{1/2}<+\infty\,.$$
    The dual space of $L^{2,\infty}(\Omega)$ is the Lorentz space $L^{2,1}(\Omega)$ whose norm is equivalent
   $$
   ||f||_{2,1}\simeq\int_0^\infty 2|\{x: |f|(x)\ge s\}|ds\,.$$

   Notice that by Fubini's Theorem, for every $\delta>0$ there exist a constant $C_\delta>0$ and a set $E^1_\delta\subset[\underline \rho,\bar\rho]$ such that $|E^1_\delta|\geq\bar\rho-\underline\rho-\delta$, and  for every $\rho\in E^1_\delta$ we have \begin{align}\label{13} \int_{\pa B_\rho}|\nabla^2 v|d\sigma\leq \frac{C_\delta}{\bar\rho-\underline\rho}\int_{B_1}|\nabla^2v|dx.\end{align}
 Recalling the continuous embedding  $W^{1,1}(\pa B_\rho)\hookrightarrow L^{2,1}(\pa B_\rho)$  (see e.g  Thm 3.3.10 in \cite{Hel}),
 the following estimate holds for every $\rho\in E^1_\delta$:   \begin{align} \|\nabla v\|_{L^{2,1}(\pa B_\rho)} &\leq  C \int_{\partial B_\rho} |\nabla^2 v|d\sigma\notag\\ &\leq  C(\delta,\bar\rho,\underline\rho)\left(\int_{B_2}e^udx+\int_{B_2}|\nabla u|^2dx\right),\label{14}\end{align} thanks to \eqref{12}-\eqref{13}.
 
Now we move on to the estimates of $\nabla v$ on the dual space $L^{2,\infty}(\pa B_\rho)$. Using the Green's representation for $v$ on $B_1$, and using that  Green's function $G$ satisfies $$|\nabla G(x,y)|\leq \frac{C}{|x-y|^2},\quad x,y\in B_1,$$ we deduce that $$|\nabla v(x)|\leq C\int_{B_1}\frac{1}{|x-y|^2}e^{u(y)}dy.$$  Then from \cite[Lemma A.2]{LT-1} we see that for every $\delta>0$ there exists $C_\delta>0$ and $E^2_\delta\subset (0,1)$ such that $|E^2_\delta|\geq 1-\delta$, and for every $\rho\in E^2_\delta$ 
\begin{align} \|\nabla v\|_{L^{2,\infty}(\pa B_\rho)}\leq C_{\delta}\int_{B_1}e^udx. \label{15} \end{align}

 Thus,  using that  $\|\nabla v\|^{2}_{L^2(\pa B_\rho)}\leq \|\nabla v\|_{L^{2,1}(\pa B_\rho)}\|\nabla v\|_{L^{2,\infty}(\pa B_\rho)}$, and  combining \eqref{14} and \eqref{15} we obtain: \begin{prop}\label{propo-grad-v} For every $\rho\in E^1_\delta\cap E^2_\delta$ we have \begin{align*} \int_{\pa B_\rho}|\nabla v|^2d\sigma \leq C(\delta,\bar \rho,\underline\rho)\left(\int_{B_1} e^u dx\right)\left( \int_{B_2}e^u dx+\int_{B_2}|\nabla u|^2dx\right). \end{align*} \end{prop}

\par
\bigskip
\noindent{\bf Proof of Theorem \ref{thm-energy-decay}} Let $\rho\in E^1_\delta\cap E^2_\delta$. 
Then from  Proposition \ref{propo-grad-v} and \eqref{est-w-3} one gets
 \begin{align*}  \int_{\pa B_\rho} |\nabla u|^2d\sigma &\leq  \int_{\pa B_\rho} |\nabla v|^2d\sigma+ \int_{\pa B_\rho} |\nabla w|^2d\sigma\notag\\ &\leq  \left(C_1\rho^2+C_2(\delta,\bar\rho,\underline\rho)\eta_1\right) \frac12\int_{B_2}|\nabla u|^2dx+C_2(\delta,\bar\rho,\underline\rho)\eta_1 \frac12\int_{B_2}  e^u dx.\end{align*}
 Hence, recalling  the following identity for stationary solutions  (see  \cite{DaLio}) 
 \begin{align} \label{stationary-identity}\frac1\rho\int_{B_\rho}\left(\frac12|\nabla u|^2-3e^u\right)dx=\frac12\int_{\pa B_\rho}|\nabla_T u|^2d\sigma-\frac12\int_{\pa B_\rho} |\pa _\nu u|^2d\sigma -\int_{\pa B_\rho}e^ud\sigma ,\end{align}  
  thanks to Proposition \ref{propo-e^u} we obtain for every $\lambda\geq 1$ the following estimate:
  \begin{align*}  
 {\mathcal{E}}_{0,\rho}(u)&=\frac{1}{\rho}\int_{B_{\rho}(x_0)}|\nabla u|^2 dx+\frac{1}{\rho}\int_{B_{\rho}(x_0)} e^u dx  \\
 &\leq \left(C_1\rho^2+\frac{2}{\rho\lambda}+C_2 \eta_1\right) \frac12\int_{B_2}|\nabla u|^2dx   \\ 
  & \quad+ \left( C_2 \eta_1+\frac{C_3 e^\lambda  \eta_1}{\rho} +\rho^2  e  \right)\frac12 \int_{B_2}e^udx. 
  \end{align*} 
   Here  $C_1$ and $C_3$ are  dimensional constants, and only $C_2$ depends on $\delta,\bar\rho,\underline\rho$. 
 
 Now for a given $\gamma\in(0,1)$ and $\eta_1\leq\bar\eta_1\leq 1$ (will be chosen later) we first fix $\bar \rho\leq\frac12$ such that $$C_1\bar\rho^2\leq\frac\gamma6\quad\text{and }\bar\rho^2 e \leq\frac\gamma6.$$  We simply take $\underline\rho:=\frac13\bar\rho$, and $\delta:=\frac{1}{10}\bar\rho $,     show that $E^1_\delta\cap E^2_\delta\not=\emptyset$.  Thus, the constant $C_2$ is fixed. Then we chose $\lambda\geq 1$ satisfying $$\frac{2}{\underline\rho\lambda}\leq\frac\gamma6,$$ and finally we chose $\bar\eta_1\in (0,1)$ satisfying $$C_2\bar\eta_1\leq\frac\gamma6\quad\text{and }\frac{C_3 e^\lambda \bar\eta_1}{\underline\rho}\leq\frac\gamma6.$$ This finishes the proof.
 \hfill $\square$
   
   \begin{lem} \label{lem-decay} Let  $\bar\eta_1>0, \underline\rho,\bar\rho$ be  as in Theorem \ref{thm-energy-decay} for some fixed $\gamma\in(0,1)$. Let $x_0\in \Omega$ and $r_0>0$ be such that  $E_{x_0,2r_0}(u)\leq \frac{\bar\eta_1}{4}$.  Then there exists $\theta\in(0,1)$ and $C>0$ (depending only on $\gamma,\underline\rho,\bar\rho$) such that $$E_{\xi,r}(u)\leq C r^\theta E_{x_0,2r_0}(u)\quad\text{for every }\xi\in B_{r_0}(x_0),\, \, 0<r\leq r_0.$$   \end{lem}
   \begin{proof} For $\xi\in B_{ r_0}(x_0) $ we set 
   \begin{align}\label{uxi}u_{\xi}(x)=u(\xi+\frac{r_0}{2} x)+2\log \frac{r_0}{2},\quad x\in B_2.\end{align} Then  $$\frac12\int_{B_2}e^{u_\xi}dx\leq\bar\eta_1.$$ Hence, by Theorem \ref{thm-energy-decay}, we can find $\rho_0\in [\underline\rho,\bar\rho]$ such that $$E_{0,\rho_0}(u_\xi)\leq \gamma E_{0,2}(u_\xi).$$ In fact, by a repeated use of  Theorem \ref{thm-energy-decay} (we just choose one $\rho_k\in[\underline\rho,\bar\rho]$ in each step), and  inductively setting   
   \begin{equation*}%\label{induct}
   u^k_{\xi}(x):=u_\xi^{k-1}(\frac{\rho_{k-1} }{2}x)+2\log\frac{\rho_{k-1}}{2} \quad \text{for }k\geq 1,\quad u^0_\xi:=u_\xi,\end{equation*}
      we have $$E_{0,\rho_k}(u_\xi^k)\leq\gamma E_{0,2}(u^{k}_\xi)=\gamma E_{0,\rho_{k-1}}(u^{k-1}_\xi)\quad\text{for every }k\geq 1.$$  
   Consequently, %{\color{blue} you are right.  replaced $R_k$ by $2^{-k}R_k$. We need to divide by $2$ in \eqref{induct} because we are applying Theorem \ref{thm-energy-decay} where $B_2$ appears on RHS, {\color{red}{OK}} }
   \begin{equation}\label{conseq} 
   E_{0,\rho_k}(u^k_\xi)\leq\gamma^k E_{0,\rho_0}(u^0_\xi)\leq\gamma^{k+1} E_{0,2}(u_\xi)\quad\Rightarrow\quad E_{0,2^{-k}R_k}(u_\xi)\leq\gamma^{k+1} E_{0,2}(u_\xi),
   \end{equation}
    where $$R_k:=\prod_{j=0}^k\rho_j\in[(\underline\rho)^{k+1}, (\bar\rho)^{k+1}],\quad R_0:=\rho_0.$$  Notice that for $0<r\leq \underline \rho$, we can find $k\geq1$ such that $$\frac{(\underline \rho)^{k+1}}{2^k}\leq \frac{R_k}{2^k}<r\leq\frac{R_{k-1}}{2^{k-1}}.$$ Then by \eqref{conseq} $$E_{0,r}(u_\xi)\leq \frac{2^k}{R_k}\frac{R_{k-1}}{2^{k-1}}E_{0, 2^{1-k}R_{k-1}}(u_\xi)=\frac{2}{\rho_k}E_{0, 2^{1-k}R_{k-1}}(u_\xi)\leq \frac{2}{\underline\rho}\gamma^k E_{0,2}(u_\xi).$$ Hence, the lemma would follow if we show that $$\frac{2}{\underline \rho}\gamma^k\leq C r^\theta, \quad  \theta=\frac{\log\gamma}{\log(\underline\rho/2)}>0,  $$ for some $C>0$ depending only on $\underline \rho$. This would follow easily if we show that $$\frac{2}{\underline \rho}\gamma^k\leq C \left(\frac{(\underline\rho)^{k+1}}{2^k}\right)^\theta=C(\underline \rho)^\theta \left(\underline  \rho/2\right)^{k\theta},$$ which is equivalent to $$ \log\frac{2}{\underline \rho}+k\log\gamma\leq\log C+\theta\log\underline\rho+k\theta\log(\underline\rho/2) .$$ The last assertion is true if we choose $C>0$ large enough so that $$ \log\frac{2}{\underline \rho} \leq\log C+\theta\log\underline\rho  .$$ We conclude the lemma.
            \end{proof}

\medskip

\noindent\textbf{Proof of Theorem \ref{thm-1}}   For $\xi\in B_{r_0}(x_0)$ we  let  $u_\xi$ be as in \eqref{uxi}. We will show that $u_\xi(0)\leq C E_{x_0,2r_0}(u)$. For this  purpose we write   $u_\xi=v_\xi+w_\xi$, where $v_\xi$ and $w_\xi$ are defined  as in \eqref{decompose} with $u=u_\xi$. By Green's representation formula we get that
\begin{equation}\label{estv0}
v_\xi(0)\leq C\int_{B_1}\frac{1}{|x|}e^{u_\xi(x)}dx.\end{equation} Now to estimate the above  integral  we shall use Lemma \ref{lem-decay}. First we fix $\gamma=\frac12$,  so that the the constant $\theta>0$ is fixed, and we have ($C$ is independent  of $u$, $\xi$, $r_0$)
 \begin{equation*} 
 E_{0,r}(u_\xi)\leq Cr^\theta E_{0,2r_0}(u_\xi).\end{equation*} From this we deduce 
 \begin{eqnarray*}
 \int_{B_1}\frac{1}{|x|}e^{u_\xi(x)}dx&=&\sum_{j=-\infty}^{0}\int_{B_{2^{j}}\setminus B_{2^{j-1}}}\frac{1}{|x|}e^{u_\xi(x)}dx\\
 &\leq&\sum_{j=-\infty}^{0}2^{-j+1}\int_{B_{2^{j}}\setminus B_{2^{j-1}}}e^{u_\xi(x)}dx\\
 &\lesssim &\sum_{j=-\infty}^{0}2^{-j}\int_{B_{2^{j}}}  e^{u_\xi(x)}dx \le \sum_{j=-\infty}^{0}(2^{j})^{\theta}E_{0,2}(u_{\xi})\leq CE_{0,2}(u_{\xi}).
 \end{eqnarray*}
  As $w_\xi$ is harmonic and $w_\xi<u_\xi$,  $$w_\xi(0)\leq \frac{1}{|B_1|} \int_{B_1}u_\xi dx\leq C\int_{B_1}e^{u_\xi}dx\leq CE_{x_0,2r_0}(u_\xi).$$
This proves   the first part of Theorem \ref{thm-1}. 
Rest of the proof   is standard. 
\hfill $\square$

 \medskip

We end this section by providing  a sequence of solutions $(u_k)$ on $B_1$ such that $$\sup_{B_1}u_k\to-\infty\quad\text{and }|\nabla u_k|\to\infty.$$ To this end, for $k\geq 1$  we define a compact operator $T_k$ on  the space $C^0(\bar B_1))$ by $T_k v=\bar v$, where $$\bar v(x):=c_0\int_{B_1}\frac{1}{|x-y|} e^{v(y)+ky_1-k^2}dy,$$ where $c_0>0$ is a dimensional constant so that $\frac{c_0}{|x|}$ is a fundamental solution of $-\D$. It follows easily that $\|T_k v\|_{C^0(\bar B_1)}\leq 1$ for $\|v\|_{C^0(\bar B_1)}\leq 1$ and $k>>1$. In particular, for $k>>1$ the operator $T_k$ has a fixed point say $v_k$. Setting  $u_k:=v_k+kx_1-k^2$ we see that this sequence has  all the desired properties.

 \section{Proof of Theorem \ref{thm-2}} 
 
Since \eqref{bound-int} and \eqref{bound-grad} hold, up to a subsequence, we have that $$e^{u_k}\rightharpoonup \mu_1,\quad |\nabla u_k|^2\rightharpoonup \mu_2,$$ for some finite Radon measures $\mu_1$ and $\mu_2$. We set $$\Sigma_1:=\left\{ x\in \Omega: \limsup_{r\to0}\frac{\mu_1(B_r(x))}{r}>0\right\},$$  $$\Sigma_2:=\left\{ x\in \Omega: \liminf_{r\to0}\frac{\mu_2(B_r(x))}{r}=\infty\right\}.$$   
Since $\mu_1(\Omega)<\infty$, by standard arguments one gets that the Hausdorff dimension of $\Sigma_1$ is at most $1$, and that $\mathcal H^1(\Sigma_2)=0$.

\begin{prop}\label{propo-seq-regu}  Let $x_0\in \Omega\setminus\Sigma, $  $\Sigma:=\Sigma_1\cup \Sigma_2$. Then $x_0$ is regular.  \end{prop}
\begin{proof} For notational convenience we take $x_0=0$.  By definition, there exists $r_i\to0$ such that for some $M\geq 1$ we have $\frac{1}{ r_i}\mu_2(B_{r_i})\leq M.$  Hence, on one hand,  there exists $N_i\in\mathbb N$ such that 

\begin{equation*}\frac{1}{r_i}\int_{B_{r_i}} |\nabla u_k|^2dx\leq 2M\quad\text{for }k\geq N_i.
\end{equation*}
 On the other hand, as $x_0\not\in\Sigma_1$, we see that $\frac{1}{r_i}\mu_1(B_{r_i})\leq \ve_i$ for some $\ve_i\to0$. Consequently, possibly replacing  $N_i$ by a larger one,  we have 
\begin{equation*}\frac{1}{r_i}\int_{B_{r_i}}e^{u_k}dx\leq 2\ve_i\quad\text{for }k\geq N_i.
\end{equation*}
 We set $$\psi_{i,k}(x)=u_k(\frac{r_i}{2}x)+2\log\frac{r_i}{2},\quad x\in B_2,\, k\geq N_i.$$  Then 
\begin{equation*} 
E_{0,2}(\psi_{i,k})\leq 3M,  \quad \frac12 \int_{B_2} e^{\psi_{i,k}} dx\leq 2\ve_i \quad\text{for }k\geq N_i.\end{equation*}
%\footnote{\color{red} In \eqref{psiM} is $2M$ or $3M$? \color{blue} the constant should be bigger than $2M$, because $\frac{1}{r_i}\int_{B_{r_i}} |\nabla u_k|^2dx\leq 2M$, right?{\color{red} You are right: I was thinking only at the Dirichlet energy}}
  Let $\eta\in(0,1)$ be as in Theorem \ref{thm-1}, and  let $\bar\eta_1$ be as in Theorem \ref{thm-energy-decay} corresponding  to the choice  of  $\gamma:=\frac{\eta}{3M}$.   We fix  $i_0>>1$  satisfying   $\ve_{i_0}\leq\bar \eta_1$.   Then setting   $N_0:=N_{i_0}$, by Theorem \ref{thm-energy-decay}, we get that $$ E_{0,\rho_{k}}(\psi_{i,k})\leq \gamma E_{0,2}(\psi_{i,k})\leq\eta\quad\text{for every }k\geq N_0,$$ for some $0<\underline\rho<\rho_{k}<\bar\rho<1$.    Now we can apply Theorem \ref{thm-1} for the sequence $(\psi_{i_0,k})_{k\geq N_0}$.  Going back to the sequence $(u_k)$ we get that $x_0=0$ is a regular point. 
\end{proof}
\par
\medskip
Now we start the 
{\bf Proof of Theorem \ref{thm-2}}. It  follows from Proposition \ref{propo-seq-regu}   that the set   $\mathcal O:=\Omega\setminus \Sigma$ is open. 
   Writing $u_k=v_k+w_k$ on smooth domains $\tilde\Omega\Subset\Omega$, with $$- \D v_k=e^{u_k}\quad\text{in }\tilde\Omega,\quad v_k=0\quad\text{on }\pa\tilde\Omega, $$ we see that $(v_k)$ is bounded (for $k$ large) in $C^2_{loc}(\tilde\Omega\setminus\Sigma)$.  Moreover, as $$\D w_k=0\quad\text{in }\tilde\Omega,\quad \int_{\tilde\Omega}w_kdx\leq C,$$  one of the following holds:  \begin{itemize} \item[(i)] $w_k\to-\infty$ locally uniformly in $\tilde\Omega$, \item[(ii)] $(w_k)$ is bounded in $C^2_{loc}(\tilde\Omega)$. \end{itemize}
  
Consequently, in case $(i)$, $u_k\to-\infty$ locally uniformly in $\Omega\setminus\Sigma$. In this case we first show that $\mu_1\equiv 0$, and in particular we get that $\Sigma_1=\emptyset$. To prove this we need the following   uniform bound  \begin{align*}%\label{strong-int}
\int_{\tilde \Omega}u_k^+e^{u_k}dx\leq C(\tilde\Omega)\quad\text{for every }k\geq 1, \, \tilde\Omega\Subset\Omega,\end{align*} which  follows from \eqref{uplusexp} of Lemma \ref{lem-2.5}, thanks to  the hypothesis \eqref{bound-grad}. Thus, the sequence  $(e^{u_k})$ is equi-integrable in $\tilde\Omega$, and consequently $$\int_{\tilde\Omega}e^{u_k}dx\to0\quad\Rightarrow \mu_1\equiv0.$$ Now we assume by contradiction that $\Sigma_2\neq\emptyset$. For $x_0\in\Sigma_2$ we fix $r_0>0$ small so that $B_{r_0}(x_0)\subset\Omega$. Then, as $$\lim_{k\to\infty}\frac{1}{r_0}\int_{B_{r_0}(x_0)}e^{u_k}dx=0,\quad \frac{1}{r_0}\int_{B_{r_0}(x_0)}|\nabla u_k|^2dx\leq C(r_0), $$ applying   Theorem  \ref{thm-energy-decay} we see that we can find $\rho\in (0,r_0)$ such that $$E_{x_0,2\rho}(u_k)\leq\eta\quad\text{for }k>>1,$$ where $\eta$ is as in Theorem \ref{thm-1}. Consequently, by Theorem \ref{thm-1}, $x_0$ is a regular point, and therefore, $x_0\not\in\Sigma_2$, a contradiction. 
 
In case $(ii)$,  for some $u\in H^1(\Omega)$ we have that $u_k\to u$ in $C^2_{loc}(\Omega\setminus \Sigma)$. Moreover, because of the equi-integrability of $(e^{u_k})$, we also have that $$e^{u_k}\to e^u\quad\text{in }L^1_{loc}(\Omega).$$ This shows that the concentration  set $\Sigma_1$ can also be written as $$\Sigma_1:=\left\{x\in\Omega:\limsup_{r\to0}\frac1r\int_{B_r(x)}e^{u(y)}dy>0\right\}.$$
 Therefore, $\mathcal H^1(\Sigma_1)=0$. 
 
 We conclude the proof of Theorem \ref{thm-2}.  
  \hfill $  \square $

     %%%%%%%%%%%%%%%%%%%%%%%%%%%%%%%%%%%%%%%%%%%%%%%%%%%%%%%%%%%%%%%%%%%%%%%%%%%%%%%%%%%%%%%%%%%

\section{Proof of Theorem \ref{thm-3}} 

The following monotonicity formula is crucial in understanding asymptotic behavior of singular solutions near the singularity.

   \subsection{Monotonicity formula}

For a weak solution $u$ to \eqref{eq-1} and $B_r(x_0)\subset \Omega$ we set    (compare \cite{HY})   \begin{align}\label{mono-energy}{\mathcal{E}}(u,x_0,r)=\frac{1}{r}\int_{B_r(x_0)}\left(\frac12|\nabla u|^2-e^u\right)dx+\frac{2}{r^{2}}\int_{\pa B_r(x_0)}(u+2\log r)d\sigma.\end{align}

\begin{prop}\label{propo-mono} Let $u$ be a stationary solution to \eqref{eq-1}.   Then the above energy  $\mathcal E$ is monotone increasing in $r$, and for $0<r<R$ with $B_R(x_0)\subset\Omega$ we have \begin{align}\label{mono-formula}{\mathcal{E}}(u,x_0,R)-{\mathcal{E}}(u,x_0,r)=\int_{r}^R\frac 1{t}\int_{\pa B_t(x_0)} \left(  \pa_\nu u+\frac2t\right)^2d\sigma dt.\end{align}\end{prop} 
\begin{proof} Up to a translation we can assume that  $x_0=0$. %We recall the following identities obtained in \cite{DaLio}:  \begin{align*} \frac1r\int_{B_r}\left(\frac12|\nabla u|^2-3e^u\right)dx=\frac12\int_{\pa B_r}|\nabla_T u|^2d\sigma-\frac12\int_{\pa B_r} |\pa _\nu u|^2d\sigma -\int_{\pa B_r}e^ud\sigma , \end{align*}  which is equivalent to \begin{align*} \frac{d}{dr}\left[  \frac1r\int_{B_r}\left(\frac12|\nabla u|^2-3e^u\right)dx  \right] =\frac1r \int_{\pa B_r} |\pa _\nu u|^2d\sigma -\frac2r\int_{\pa B_r}e^ud\sigma. \end{align*} 
By  \eqref{rela-stationary},  and  \begin{align*}-\frac{1}{r^2}\int_{B_r}e^udx&= \frac{d}{dr}\left(\frac{1}{r^2}\int_{\pa B_r}u d\sigma\right) =\int_{\pa B_1}\frac{d}{dr} u(rx)d\sigma=\frac{1}{r^2}\int_{\pa B_r}\pa_\nu u\, d\sigma,\end{align*} we obtain    \begin{align*} \frac{d}{dr} {\mathcal{E}}(u,0,r)&=\frac1r \int_{\pa B_r} |\pa _\nu u|^2d\sigma -\frac2r\int_{\pa B_r}e^ud\sigma+\frac{d}{dr}\left [ \frac2r\int_{B_r} e^udx \right]  \\ &\quad -\frac{2}{r^2}\int_{B_r}e^udx+\frac{4}{r}|\pa B_1| \\&=\frac1r \int_{\pa B_r} \left( |\pa _\nu u|^2+\frac{4}{r^2}\right)d\sigma  -\frac{4}{r^2}\int_{B_r}e^udx \\&=\frac1r\int_{\pa B_r}\left(\pa_\nu u+\frac 2r\right)^2d\sigma.\end{align*}This proves the proposition.\end{proof}

\subsection{Asymptotic behavior of singular solutions}

Let $u$ be a stationary solution to \eqref{eq-1} satisfying \eqref{Morrey-bound}.  Then  by Theorem \ref{thm-1} we see that the singular set $\Sigma$ is closed in $\Omega$, and  $\HH^1(\Sigma)=0$. In our next lemma we prove a lower bound of $u$ around each point in $\Sigma$. 

\begin{lem}  \label{lem-lower}  Let $x_0\in \Sigma$. Then \begin{align}\label{25}\liminf_{r\to0} \inf_{\pa B_r(x_0)} \left( u+2\log r\right) >-\infty.    \end{align}  \end{lem}
\begin{proof} We assume by contradiction that \eqref{25} is false for some $x_0\in\Sigma$.   Then there exists $r_k\to0$ such that $$\lim_{k\to\infty} \inf_{\pa B_{r_k}(x_0)} \left( u+2\log r_k\right) =-\infty.$$ We set $$u_k(x):= u(x_0+r_kx)+2\log r_k.$$ Then for $k$ large, $u_k$ satisfies $$-\D u_k=e^{u_k}\quad\text{in }B_2, \quad \int_{B_2}|\nabla u_k|^2dx\leq C,$$ thanks to \eqref{Morrey-bound}. Moreover, as $u_k$ also satisfies $$ \inf_{\pa B_1} u_k  \to-\infty,$$  we see that only  Case 1 of Theorem \ref{thm-2} can occur. Consequently, $u_k\leq C$ in $B_1$ for $k$ large. This in turn implies that    $u\leq C$ in $B_\delta(x_0)$ for some $\delta>0$ small, a contradiction to the hypothesis that $x_0\in\Sigma$. 
\end{proof}

The above lemma and the monotonicity formula leads us to the following result, which will be used later.

\begin{lem} \label{lem-u-nu}Let $x_0\in\Sigma$. Then for $r_0>0$ with $B_{2r_0}(x_0)\subset\Omega$ we have $$\int_{0}^{r_0}\frac1r\int_{\pa B_r(x_0)}\left(\pa_\nu u+\frac2 r\right)^2d\sigma dr<\infty.$$  \end{lem}
\begin{proof}  For simplicity we take $x_0=0$ and $r_0=1$.  By Lemma \ref{propo-mono} we get that the function   $$r\mapsto  {\mathcal{E}}(u_k,r):={\mathcal{E}}(u_k,x_0,r),\quad r\in (0,1)$$ is monotone increasing, where the energy $\mathcal E$ is as defined in \eqref{mono-energy}. By \eqref{Morrey-bound} one gets   $$\frac1r\int_{B_r(x_0)}(|\nabla u|^2+e^{u})dx\leq C\quad\text{for }r\in(0,1].$$ Indeed,  for $ r\in (0,1]$, fixing a  nonnegative test function   $\vp\in C_c^\infty(B_{2r}(x_0))$  satisfying \begin{align*}  \vp\equiv 1\quad\text{in }B_r(x_0),\quad |\nabla \vp|\leq \frac2r,  \end{align*}  we obtain  \begin{align*}\frac1r\int_{B_r(x_0)}e^{u}dx & \leq\frac1r \int_{B_{2r}(x_0)} e^{u}\vp dx  =\frac1r\int_{B_{2r}(x_0)} \nabla \vp\cdot\nabla udx \\&\leq \frac {C}{\sqrt r}\left(\int_{B_{2r}(x_0)}|\nabla u|^2dx\right)^\frac12=C\left(\frac{1}{2r}\int_{B_{2r}(x_0)}|\nabla u|^2dx\right)^\frac12\leq C. \end{align*}    

From Proposition \ref{propo-mono}, for every $t\in(0,1)$ \begin{align*}  \int_t ^1\frac1r\int_{\pa B_r(x_0)}\left (\pa_\nu u+\frac 2r\right) ^2 d\sigma  dr&={\mathcal{E}}(u,x_0, 1)-{\mathcal{E}}(u,x_0,t) \\&  \leq C-  \frac{2}{t^2}\int_{\pa B_t(x_0)} (u+2\log t)d\sigma.  \end{align*}   It follows from \eqref{lem-lower} that $$-\frac{1}{t^2}\int_{\pa B_t(x_0)} (u+2\log t)d\sigma\leq C\quad\text{for every }t\in(0,1].$$ The lemma follows as $t>0$ is arbitrary. 
 \end{proof}

Now we prove an upper bound of $u$ around each point on $\Sigma$. 

\begin{lem} \label{lem-upper}  Let $x_0\in \Sigma$. 
Then \begin{align}
\label{26}\limsup_{r\to0} \emph{ ess}\sup_{\pa B_r(x_0)} \left( u+2\log r\right) <+\infty.  
  \end{align}  
   \end{lem}
\begin{proof} We assume by contradiction that \eqref{26} is false for some $x_0\in\Sigma$.   Then there exists $r_k\to0$ such that $$\lim_{k\to\infty} \emph{ ess}\sup_{\pa B_{r_k}(x_0)} \left( u+2\log r_k\right) =+\infty.$$ We set $$u_k(x):= u(x_0+r_kx)+2\log r_k.$$ Then for $k$ large, $u_k$ satisfies  (on domains invading $\R^3$) $$-\D u_k=e^{u_k},\quad \quad \frac 1r \int_{B_r}|\nabla u_k|^2dx\leq C,$$ thanks to \eqref{Morrey-bound}. Since $x_0\in\Sigma$,  by Theorem \ref{thm-2} we get that, up to a subsequence, $u_k\to u_\infty$ in $C^2_{loc}(\R^3\setminus \Sigma_\infty)$ for some closed set $\Sigma_\infty$ with $\HH^1(\Sigma_\infty)=0$.  It follows that \begin{align}\label{27}0\in\Sigma_\infty\quad\text{and }\Sigma_\infty\cap \pa B_1\neq\emptyset.  \end{align}

{\bf Claim:} $u_\infty$ is homogeneous, that is $$u_\infty(rx)+2\log r=u_\infty(x).$$ This and \eqref{27} would imply that $\HH^1(\Sigma_\infty)\geq1$, a contradiction. 

{\bf Proof of the Claim:}  We consider the monotonicity formula \eqref{mono-formula} from Proposition \ref{propo-mono}. Notice that $${\mathcal{E}}(u_k,0,r)={\mathcal{E}}(u,x_0, rr_k),$$ where the energy ${\mathcal{E}}$ is defined in \eqref{mono-energy}.  Moreover we have $${\mathcal{E}}(u,x_0,r)\geq -C\quad\text{for }r>0,$$ thanks to \eqref{Morrey-bound} and Lemma \ref{lem-lower}. Therefore, by the monotonicity of  $r\mapsto {\mathcal{E}}(u,x_0,r)$, the following limit exists finitely: $$\ell:= \lim_{r\to0} {\mathcal{E}}(u,x_0,r).$$ In particular, for any $0<r<R$, we have $$\lim_{k\to\infty} \left({\mathcal{E}}(u_k,0,R)-{\mathcal{E}}(u_k,0,r)\right)=\ell-\ell=0.$$ Therefore, by Fatou's lemma, and from \eqref{mono-formula} \begin{align*}%\label{monot2} 
 \int_r^R \frac1t\int_{\pa B_t }\left(\pa_\nu u_\infty +\frac 2t\right)^2d\sigma dt&\leq  \liminf_{k\to\infty}  \int_r^R \frac1t\int_{\pa B_t }\left(\pa_\nu u_k +\frac 2t\right)^2d\sigma dt\\&=\liminf_{k\to\infty} \left({\mathcal{E}}(u_k,0,R)-{\mathcal{E}}(u_k,0,r)\right) \\&=0.\end{align*}
 We  conclude the proof of the Claim and of the Lemma \ref{lem-upper}.
\end{proof}

  \medskip 
 
 \noindent{\bf Proof of \eqref{beha-singu} and \eqref{int}.}  Assume that the singular set $\Sigma$ is non-empty. Then from Lemmas \ref{lem-lower} and \ref{lem-upper} we get that for some $\delta>0$ $$u(x)+2\log|x-x_0|=O(1)\quad\text{for }|x-x_0|<\delta,$$ which is  \eqref{beha-singu}. In order to prove \eqref{int} we take any sequence $(r_k)$ with $r_k\to0^+$. Then setting $$u_k(x)=u(x_0+r_k x)+2\log r_k,$$ we get that, up to a subsequence, $$ u_k\to u_\infty\quad C^2_{loc} (\R^3\setminus\{0\}),\quad e^{u_k}\to e^{u_\infty}\quad\text{in }L^1_{loc}({\R^3}).$$ Moreover, (as in the proof of Lemma \ref{lem-upper}) $u_\infty$ is a homogeneous solution of \eqref{eq-1} in $\R^3$. Therefore, $$\lim_{k\to\infty} \frac{1}{r_k} \int_{B_{r_k}(x_0)}e^u dx=\lim_{k\to\infty}\int_{B_1} e^{u_k}dx=\int_{B_1}e^{u_\infty} dx=8\pi,$$ thanks to Lemma \ref{lem-entire-2}. 
 
  \hfill $\square$
 
 \medskip

\noindent{\bf Proof of Corollary \ref{cor-1}}  Setting $\Sigma:=\Omega\setminus\mathcal O$ we see that $\Sigma$ is closed in $\Omega$ with $\HH^1(\Sigma)=0$ and (up to a subsequence) $u_k\to u  $ in $C^2_{loc}(\Omega\setminus\Sigma)$. Let $U\Subset\Omega$ be an open set.  Since $\HH^1(\Sigma)=0$ and $\tilde\Sigma:=\bar U\cap \Sigma$ is compact,    for a given $\ve>0$ we can find   finitely many balls   $B_{4r_i}(x_i)\subset\Omega$ such that $$\tilde\Sigma \subset \cup_{i=1}^ N B_{r_i}(x_i),\quad \sum_{i=1}^N r_i\leq\ve. $$    Setting  $A=\cup_{i=1}^N B_{4r_i}(x_i)$ one obtains $$\int_{A}|\nabla u_k|^2dx\leq \sum_{i=1}^N \int_{B_{4r_i}(x_i)} |\nabla u_k|^2 dx\leq 4M\sum_{i=1}^ N r_i\leq 4M\ve.$$  Notice that $\tilde U:=U\setminus(\cup_{i=1}^N \bar B_{2r_i}(x_i))\Subset\Omega\setminus\Sigma$, and therefore, $u_k\to u$ in $H^1(\tilde U)$. Since $\ve>0$ is arbitrary, we conclude that $u_k\to u$ in $H^1(U)$. 

In order to show that the concentration set $\Sigma$ is discrete in $\Omega$, we first note that the   limit function $u$ is  a stationary solution.   Hence, by Theorem \ref{thm-3}, the singular set $\Sigma_u$ of $u$ would be discrete. Now if we take $x_0\in\Omega\setminus \Sigma_u$, then we can choose $r>0$ small enough such that $B_r(x_0)\cap\Sigma_u=\emptyset$, and (as $ u$ is regular around $x_0$) $$ E_{x_0,r}(u) =\frac1r\int_{B_r(x_0)}(|\nabla u|^2+e^u)dx<\ve.$$ Then by the  convergences $$u_k\to u\quad\text{in }H^1_{loc}(\Omega),\quad e^{u_k}\to e^{u}\quad\text{in }L^1_{loc}(\Omega),$$ we get that $E_{x_0,r}(u_k)\leq 2\ve$ for $k>>1$. Therefore, we can use Theorem \ref{thm-1} to conclude that $u_k\leq C$ in $B_\frac r2(x_0)$, and in particular $x_0$ is not a blow-up point.  

We conclude the proof. \hfill $\square$
 
 \medskip

\noindent{\bf Proof of Theorem  \ref{thm-4}} Let us first assume that $u$ is regular. If the theorem were false then $$ |u(x_k)+2\log |x_k||\to\infty\quad\text{for some }x_k\in\R^3\text{ with }|x_k|\to\infty.$$ We define $$u_k(x):=u(|x_k|x)+2\log |x_k|. $$ It then follows that $$\int_{B_2}|\nabla u_k|^2dx\leq C,\quad u_k(0)\to\infty.$$ Therefore, by Theorem \ref{thm-2}, up to a subsequence, $u_k\to u_\infty$ in $C^2_{loc}(\R^3\setminus\Sigma)$ for some closed set $\Sigma$ with $\mathcal H^1(\Sigma)=0$.  In particular, $\max_{\pa B_1} u_k  \to\infty$, and hence $u_\infty$ would be singular at $x_\infty$ for some $x_\infty\in \pa B_1.$ As in the proof of \eqref{beha-singu} one can show that $u_\infty$ is homogeneous, a contradiction. 

We now assume that $u$ is singular, and up to a translation, origin is a singular point. One can show that (see e.g. proof of \eqref{homog} and Lemma \ref{lem-energy-origin}) that $$\lim_{r\to0}\mathcal E(u,0,r)=8\pi\log 2=\lim_{r\to\infty }\mathcal E(u,0,r).$$ Hence $u$ is homogeneous, and we conclude the proof. 
\hfill$\square$

\medskip   
 
 Next  we prove \eqref{homog}. For $r>0$ we set $$u_r(x):=u(x_0+rx)+2\log r. $$  Following the proof of Lemma \ref{lem-upper} we see that for  any sequence  $r_k\to 0$, it has a subsequence such that (still denoting it by $r_k$)    $$u_{r_k} \to u_0,$$ where $u_0$ is homogeneous. We would show that the limit function $u_0$ is independent of the choice of $(r_k)$. % , and this would complete the proof of \eqref{homog}.  
 
 \subsection{Uniqueness of the limit function}
 
 \begin{lem}\label{Cauchy}For every $\ve>0$ there exists $\delta>0$ such that \begin{align*} \int_{\pa B_{1}(x_0)}|u_{r_1}(x)-u_{r_2}(x)|^2d\sigma(x) <\ve\quad\text{for every }r_1,r_2\in (0,\delta). \end{align*}   \end{lem}
 \begin{proof}  For simplicity we assume  that $x_0=0$. Then for $|x|=1$ and $0<r<R\leq2r$ we have \begin{align*} u_R(x)-u_r(x)&=2\log \frac Rr+\int_r^R \frac{d}{dt} u(tx)dt \\ &=\int_r^R \left( \pa_{\nu} u(tx) +\frac2t\right) dt,  \end{align*} and together with H\"older's inequality \begin{align*} |u_R(x)-u_r(x)|^2\leq (R-r)\int_r^R  \left( \pa_{\nu} u(tx) +\frac2t\right) ^2dt. \end{align*} Integrating the above inequality on $\pa B_1$ \begin{align*}\int_{\pa B_1}  |u_R(x)-u_r(x)|^2d\sigma(x) &\leq (R-r)\int_r^R  \int_{\pa B_1}\left( \pa_{\nu} u(tx) +\frac2t\right) ^2d\sigma dt \notag \\ &= (R-r)\int_r^R \frac{1}{t^2} \int_{\pa B_t}\left( \pa_{\nu} u +\frac2t\right) ^2d\sigma dt.  \end{align*} Since  $R\leq 2r$ implies $\frac{R-r}{t}\leq \frac r t\leq 1$,  we deduce that \begin{align}\int_{\pa B_1}|u_R(x)-u_r(x)|^2d\sigma(x)\leq   \int_r^R \frac{1}{t} \int_{\pa B_t}\left( \pa_{\nu} u +\frac2t\right) ^2d\sigma dt. \notag   \end{align}  Consequently, for any integer $m\geq 1$, \begin{align*}\int_{\pa B_1}|u_{2^mr}(x)-u_r(x)|^2d\sigma(x) &\leq 2\sum_{k=0}^{m-1}\int_{\pa B_1}|u_{2^{k+1}r}(x)-u_{2^kr}(x)|^2d\sigma(x)\\ &\leq 2\sum_{k=0}^{m-1}\int_{2^kr}^{2^{k+1}r} \frac{1}{t} \int_{\pa B_t}\left( \pa_{\nu} u +\frac2t\right) ^2d\sigma dt \\ &=2\int_r^{2^mr} \frac{1}{t} \int_{\pa B_t}\left( \pa_{\nu} u +\frac2t\right) ^2d\sigma dt.\end{align*} The lemma follows by choosing $\delta>0$ such that $$\int_0^\delta \frac{1}{t} \int_{\pa B_t}\left( \pa_{\nu} u +\frac2t\right) ^2d\sigma dt\leq\frac\ve 8,$$ which is possible, thanks to Lemma \ref{lem-u-nu}. 
  \end{proof}
 
 \begin{rem}\label{rem-4.5} The above proof shows that the following estimate holds: for $0<r_1<r_2$ $$\int_{\pa B_1}|u_{r_2}-u_{r_1}|^2d\sigma\leq C\int_{r_1}^{r_2}\frac1t\int_{\pa B_t}\left(\pa_\nu u+\frac 2t\right)^2d\sigma dt,$$
where $C>0$ is independent of $u$ and $r_1,r_2$. 
 \end{rem}
 
\medskip

\noindent{\bf Proof of \eqref{homog}.} For $r>0$ small, setting $$ u_r(x)=u(x_0+rx)+2\log r,$$ we see that $u_r$ satisfies $$-\D u_r=e^{u_r}\quad\text{in }B_2\setminus B_\frac12, \quad u_r=O(1)\quad\text{in }B_2\setminus B_\frac12,$$ thanks to \eqref{beha-singu}. Therefore, by elliptic estimates we get that $(u_r)_{r>0}$ is bounded in $C^{1,\alpha}_{loc}(B_2\setminus B_\frac12)$. By a bootstrapping argument we see that for every integer $\ell\geq 1$, the family $(u_r)_{r>0}$ is bounded in  $C^{\ell}_{loc}(B_2\setminus B_\frac12)$. Therefore, together with Lemma \ref{Cauchy}, we get that $$u_r\xrightarrow{r\to0} u_0\quad\text{in }  C^{\ell}_{loc}(B_2\setminus B_\frac12).$$ As before, the function $u_0$ is homogeneous, that is $$u_0(tx)=-2\log t+u_0(x).$$  Therefore, we can write $$u_0(x)=-2\log|x|+\log 2+2\vp (x),$$  for some  $\vp\in C^\infty(B_2\setminus B_\frac12)$  such that  $\vp(x)=\vp(\frac{x}{|x|})$. Moreover, from the equation $-\D u_0=e^{u_0}$,  it  follows easily that \begin{align}\label{eq-vp}-\D_{S^2}\vp+1=e^{2\vp}\quad\text{in }S^2.\end{align}
\hfill $\square$

\subsection{Some useful lemmas}

\begin{lem}\label{negativeaverage} If $$\int_{B_2}e^udx<\ve,$$ for small $\ve>0$ then $$\int_{\pa B_1}ud\sigma<0, \quad \int_{B_1}udx<0. $$  \end{lem}
 \begin{proof} We set $$\bar u(r)=\frac{1}{|\pa B_r|}\int_{\pa B_r} ud\sigma.$$ Then $\bar u$ is monotone decreasing, and $$\int_{B_2} e^{\bar u} dx\leq \int_{B_2 }e^u dx<\ve.$$ Hence, $$e^{\bar u(1)}|B_1|\leq \int_{B_1}e^{\bar u}dx\leq \ve.$$ This gives that $\bar u(1)<0$ for $\ve>0$ small enough. 
 
 By Jensens inequality, \begin{align*} exp\left( \frac{1}{|B_1|}\int_{B_1} u dx\right)\leq \frac{1}{|B_1|} \int_{B_1}e^u dx<\frac{\ve}{|B_1|},  \end{align*} and hence $$\int_{B_1}u dx<0.$$
  \end{proof}

\begin{lem} \label{lem-useful-1} Let $u\in C^2(B_1\setminus\{0\})$ be a distributional solution to \eqref{eq-1} with $\Omega=B_1$. Assume that $$u(x)\leq -2\log |x|+C\quad\text{for }0<|x|<1.$$ Then $u\in H^1_{loc}( B_1)$, and for every $0<R<1$ we have $$\sup_{B_r(x_0)\subset B_R}\frac1r\int_{B_r(x_0)}|\nabla u|^2dx<\infty.$$ \end{lem}
\begin{proof} We write  $u=v+w$ where $w$ is harmonic in $B_1$, and $v$ is given by $$v(x):=c_0\int_{B_1}\frac{1}{|x-y|}e^{u(y)}dy,$$ for some dimensional constant $c_0>0$ so that $-\D v=e^u$. It then follows that $v\in C^2(B_1\setminus\{0\})$, and differentiating under the integral sign one gets  \begin{align*}|\nabla v(x)|\leq C\int_{B_1}\frac{dy}{|x-y|^2|y|^2}\leq\frac{C}{|x|},\quad x\in B_1\setminus \{0\}.   \end{align*} Therefore, for $B_r(x_0)\subset B_1$ we have $$\frac1r\int_{B_r(x_0)}|\nabla v|^2dx\leq \frac1r\int_{B_r(x_0)}\frac{dx}{|x|^2}\leq   C.$$ The lemma follows immediately as $w\in C^\infty(\bar B_R)$ for every $0<R<1$. 
\end{proof}

A simple consequence of \eqref{homog} is the following: 

\begin{lem}\label{lem-energy-origin}  We have $$\lim_{r\to0} {\mathcal{E}}(u,x_0,r)=8\pi\log 2.$$ \end{lem}
\begin{proof}  We have  \begin{align*} {\mathcal{E}}(u,x_0,r)&={\mathcal{E}}(u_r,0,1)\\ &= \int_{B_1} \left(\frac12|\nabla u_r|^2-e^{u_r}\right)dx+2\int_{\pa B_1} u_rd\sigma \\&=\int_0^1\int_{\pa B_t}\left(\frac12|\nabla u_r|^2-e^{u_r}\right)d\sigma dt +2\int_{\pa B_1} u_rd\sigma \\& =\int_0^1\int_{\pa B_1}\left(\frac12|\nabla u_{rt}|^2-e^{u_{rt}}\right)d\sigma dt +2\int_{\pa B_1} u_rd\sigma. \end{align*} Using that   $$u_r\xrightarrow{r\to0} u_0\quad\text{in }  C^{\ell}_{loc}(B_2\setminus B_\frac12), \quad u_0(x)=-2\log|x|+\log 2+2\vp (x),$$ we obtain $$|\nabla u_{rt}|^2(x)=4+4|\nabla_{S^2} \vp|^2+o_r(1)\quad\text{on }S^2.$$  Hence, as $\int_{S^2} e^{2\vp}d\sigma=4\pi$ \begin{align*} \lim_{r\to0} {\mathcal{E}}(u,x_0,r)=8\pi \log 2+2\left(\int_{S^2} |\nabla_{S^2}\vp |^2d\sigma+2\int_{S^2} \vp d\sigma \right).\end{align*} Now we recall that for any $w\in H^1(S^2)$ we have (see e.g. \cite{Onofri}) $$\int_{S^2}|\nabla w|^2\frac{d\sigma}{4\pi} +2\int_{S^2} w \frac{d\sigma}{4\pi}-\log\left(\int_{S^2} e^{2w} \frac{d\sigma}{4\pi}\right)\geq 0,$$ and $w$ is an extremum of the above inequality if and only if $w$ satisfies \eqref{eq-vp}. Therefore, again using that $\int_{S^2} e^{2\vp}d\sigma=4\pi$, we get that $$\int_{S^2} |\nabla_{S^2}\vp |^2d\sigma+2\int_{S^2} \vp d\sigma =0.$$
\end{proof}

\medskip 

The following lemma 
(precisely the estimate \eqref{thm-est-2}) suggests that Theorem \ref{thm-3} would be true without the assumption \eqref{Morrey-bound}.
 
\begin{lem} Let $u$ be  a stationary solution to \eqref{eq-1} in $B_1$. Assume that $u$ is singular at the origin, and  \begin{align*}  \sup_{0<r<1}\frac1r\int_{B_r}|\nabla u|^2 dx<\infty.\end{align*}  Then  ($C$ is independent of $u$) \begin{align}\label{thm-est-1} \int_0^\frac12 \frac{1}{r}\int_{\pa B_r}\left(\pa_\nu u+\frac2r\right)^2 d\sigma dr\leq C+C\int_{B_1}|\nabla u|^2dx.   \end{align} In particular, there exists $C>0$ (independent of $u$) such that  \begin{align}\label{thm-est-2}\sup_{0<r<\frac12}\frac1r\int_{B_r}\left(|\nabla u|^2+e^u\right)dx\leq C+C\int_{B_1}|\nabla u|^2dx+C\int_{\pa B_\frac12}|u|d\sigma. \end{align} %and $$\sup_{0<r<\frac12}\frac1r\int_{B_r} |\nabla  u|^2dx\leq C+C\int_{B_1}|\nabla u|^2dx+C\int_{\pa B_\frac12}|u|d\sigma.$$
 \end{lem}
\begin{proof} For convenience we set \begin{align*} M:= \int_{B_1}|\nabla u|^2 dx  . \end{align*} Then   we have that $$ \int_{B_\frac12}e^{u}dx\leq CM,$$ which yields $$\bar u:=\frac{1}{|B_\frac12|}\int_{B_\frac12}u dx\leq CM.$$ Hence, by the trace embedding $H^1(B_\frac12)\hookrightarrow L^1(\pa B_\frac12)$ \begin{align*}{\mathcal{E}}(u,0,\frac12)&=2\int_{B_\frac12}\left(\frac12|\nabla u|^2-e^{u}\right)dx+8\int_{\pa B_\frac12} (u-\bar u +\log \frac12 +\bar u )d\sigma \\ & \leq C(1+M).\end{align*} Therefore, by Lemma \ref{lem-energy-origin} and \eqref{mono-formula} \begin{align*} \int_0^\frac12 \frac{1}{r}\int_{\pa B_r}\left(\pa_\nu u+\frac2r\right)^2 d\sigma dr= {\mathcal{E}}(u,0,\frac12)-8\pi\log 2\leq C(1+M),   \end{align*} which is \eqref{thm-est-1}. Since $$|\pa_\nu u|^2\leq 2\left( |\pa_\nu u+\frac 2r|^2 +\frac{4}{r^2}\right) ,$$ we get  for $0<r\leq\frac12$ \begin{align*}  \frac1r\int_{B_r}|\pa_\nu u|^2dx &= \frac1r \int_0^r\int_{\pa B_t}|\pa_\nu u|^2d\sigma dt\\ &\leq  2\int_0^r \frac t  r \frac1t \int_{\pa B_t}\left(\pa_\nu u+\frac 2t\right)^2 d\sigma dt+32\pi\\ &\leq C(1+M).\end{align*}

Notice that  for any  $ r_0\in (0,\frac14)$ there exists $t_0\in (r_0, 2r_0)  $ such that ($t_0$ would depend on $u$) $$\int_{\pa B_{t_0}}|\pa _\nu u|^2d\sigma\leq \frac{1}{r_0}\int_{B_{2r_0}} |\pa _\nu u|^2dx\leq C(1+M).$$ Hence, for $r_0\in (0,\frac14)$ \begin{align*} \frac{1}{r_0 } \int_{B_{r_0}} e^{u} dx &\leq \frac{2}{t_0 } \int_{B_{t_0}} e^{u} dx= -\frac{2}{t_0 } \int_{\pa B_{t_0}}  \pa _\nu u d\sigma \\&\leq \frac{2}{t_0}\left(\int_{\pa B_{t_0}} |\pa _\nu u|^2d\sigma\right)^\frac12 |\pa B_{t_0}|^\frac12\\&\leq C(1+M), \end{align*} which gives $$\sup_{0<r\leq\frac12}\frac1r\int_{B_r}e^udx\leq C(1+M). $$

Now setting $u_r(x):=u(rx)+2\log r$,  and by Remark \ref{rem-4.5} we obtain $$\int_{\pa B_1}|u_\frac 12-u_r|^2d\sigma\leq C(1+M)\quad\text{for every }0<r\leq\frac12,$$ thanks to \eqref{thm-est-1}. This leads to $$\int_{\pa B_1}|u_r|  d\sigma\leq C(1+M)+\int_{\pa B_1}|u_\frac12|d\sigma,\quad 0<r\leq\frac12.$$ Now \eqref{thm-est-2} follows immediately from the monotonicity of the map  $$r\mapsto {\mathcal{E}}(u,0,r)={\mathcal{E}}(u_r,0,1) . $$
\end{proof}

 \begin{lem}\label{lem-entire-2} Let $u\in H^1_{loc}(\R^n)$ be a weak solution to   \begin{align*}  -\D u=e^u\quad\text{in }\R^n.\end{align*}  If $u$ is homogeneous, that is, $$u(rx)+2\log r=u(x)\quad\text{for every }r>0,$$ then \begin{align}\label{ind}\frac{1}{r^{n-2}}\int_{B_r}e^u dx=2|S^{n-1}|\quad\text{for every }r>0.\end{align} \end{lem}
 \begin{proof}   Since $u$ is homogeneous,  the integral in \eqref{ind} is independent of $r>0$, and we call it $\ell$. Then we have  $$-\frac{d}{dr}\left(\frac{1}{r^{n-1}}\int_{\pa B_r} ud\sigma\right)=\frac{1}{r^{n-1}}\int_{B_r}e^udx=\frac\ell r.$$ Integrating the above relation, $$\frac{1}{r^{n-1}}\int_{\pa B_r} ud\sigma -\int_{\pa B_1}ud\sigma=\ell\log\frac1r.$$ As $u$ is homogeneous $$ \frac{1}{r^{n-1}}\int_{\pa B_r} ud\sigma=\int_{\pa B_1} u(rx)d\sigma(x)=\int_{\pa B_1}(u(x)-2\log r)d\sigma(x)=\int_{\pa B_1}ud\sigma +2|S^{n-1}|\log\frac1r.$$
 Combining the above two relations we  obtain \eqref{ind}. 
 \end{proof}

\medskip 

{\bf \noindent Acknowledgement:}
The work of Ali Hyder is partially supported by SNSF grant no. P4P4P2-194460 and SERB   SRG/2022/001291.

            \end{document}